\documentclass[12pt,a4paper,fleqn]{article}
\usepackage{a4wide,amsfonts,amsmath,latexsym,amssymb,euscript,graphicx,units,mathrsfs}

\usepackage{graphicx}
\usepackage{color}
\usepackage{amssymb}
\usepackage{amssymb}
\usepackage[T1]{fontenc}
\usepackage{latexsym}
\usepackage{xypic}
\usepackage{eufrak}
\usepackage{euscript}
\usepackage{amsfonts,amsmath}
\usepackage{verbatim}
\usepackage[english]{babel}
\usepackage{mathrsfs}
\usepackage{units}

\newtheorem{prop}{Proposition}[section]

\newtheorem{lemma}[prop]{Lemma}
\newtheorem{rem}[prop]{Remark}
\newtheorem{remark}[prop]{Remark}
\newtheorem{thm}[prop]{Theorem}
\newtheorem{theorem}[prop]{Theorem}

\newtheorem{definition}[prop]{Definition}

\renewcommand{\geq}{\geqslant}
\def\leq{\leqslant}

\newcommand{\R}{\mathbb{R}}

\def\HH{\EuFrak H}

\def\1{{\mathbf{1}}}

\def\1{{\mathbf{1}}}
\def\0.5{{\frac{1}{2}}}

\newcommand{\fin}
{ \vspace{-0.6cm}
\begin{flushright}
\mbox{$\Box$}
\end{flushright}
\noindent }

\newenvironment{proof}[1]{\begin{trivlist}\item {\it
\bf Proof.}\quad} {\qed\end{trivlist}}

\newcommand{\qed}{\nopagebreak\hspace*{\fill}
{\vrule width6pt height6ptdepth0pt}\par}

\begin{document}

\begin{center}
{\LARGE{\bf Cumulants on the Wiener Space}}\\~\\
by Ivan Nourdin\footnote{Laboratoire de Probabilit{\'e}s et
Mod{\`e}les Al{\'e}atoires, Universit{\'e} Paris VI,
Bo{\^\i}te courrier 188, 4 Place Jussieu, 75252 Paris Cedex 5,
France. Email: {\tt ivan.nourdin@upmc.fr}} and Giovanni
Peccati\footnote{Equipe Modal'X, Universit\'{e} Paris Ouest -- Nanterre la D\'{e}fense, 200 Avenue de la République, 92000 Nanterre, and LSTA, Universit\'{e} Paris VI, France. Email: \texttt{giovanni.peccati@gmail.com}} \\
{\it Universit\'e Paris VI and Universit\'e Paris Ouest}\\~\\
\end{center}
{\small \noindent {\bf Abstract:} We combine infinite-dimensional integration by parts procedures with a recursive relation on moments (reminiscent of a formula by Barbour (1986)), and deduce explicit expressions for cumulants of functionals of a general Gaussian field. These findings yield a compact formula for cumulants on a fixed Wiener chaos, virtually replacing the usual ``graph/diagram computations'' adopted in most of the probabilistic literature.\\


 }
\\

\section{Introduction}
The {\sl integration by parts formula} of Malliavin calculus, combining derivative operators and anticipative integrals into a flexible tool for computing and assessing mathematical expectations, is a cornerstone of modern stochastic analysis. The scope of its applications, ranging e.g. from density estimates for solutions of stochastic differential equations to concentration inequalities, from anticipative stochastic calculus to ``Greeks'' computations in mathematical finance, is vividly described in the three classic monographs by Malliavin \cite{MallBook}, Janson \cite{Janson} and Nualart \cite{nualartbook}.

In recent years, infinite-dimensional integration by parts techniques have found another fertile ground for applications, that is, {\sl limit theorems} and (more generally) {\sl probabilistic approximations}. The starting point of this active line of research is the paper \cite{NO}, where the authors use Malliavin calculus in order to refine some criteria for asymptotic normality on a fixed Wiener chaos, originally proved in \cite{nunugio, PTu04} (see also \cite{noncentral} for some non-central version of these results). Another important step appears in \cite{NP-PTRF}, where integration by parts on Wiener space is combined with the so-called {\sl Stein's method} for probabilistic approximations (see e.g. \cite{chen-shao, Reinert_sur}), thus yielding explicit upper bounds in the normal and gamma approximations of the law of functionals of Gaussian fields. The techniques introduced in \cite{NP-PTRF} have led to several applications and generalizations, for instance: in \cite{exact} one can find applications to Edgeworth expansions and reversed Berry-Esseen inequalities; \cite{multivariate} contains results for multivariate normal approximations; \cite{AirMalVie} focuses on further developments in the multivariate case, in relation with quasi-sure analysis; \cite{Noupecrei2} deals with infinite-dimensional second order Poincar\'{e} inequalities; in \cite{NV}, one can find new explicit expressions for the density of functionals of Gaussian field as well as new concentration inequalities (see also \cite{breton-nourdin-peccati} for some applications in mathematical statistics); in \cite{Noupecrei3}, the results of \cite{NP-PTRF} are combined with Lindeberg-type invariance principles in order to deduce {\sl universality results} for homogeneous sums (these findings are further applied in \cite{NouPeMATRIX} to random matrix theory).

The aim of this note is to develop yet another striking application of the infinite-dimensional integration by parts formula of Malliavin calculus, namely the computation of {\sl cumulants} for general functionals of a given Gaussian field. As discussed below, our techniques make a crucial use of a recursive formula for moments (see relation (\ref{EQ : RecMom}) below), which is the starting point of some well-known computations performed by Barbour in \cite{BarbourPtrf1986} in connection with the Stein's method for normal approximations. As such, the techniques developed in the forthcoming sections can be seen as further ramifications of the findings of \cite{NP-PTRF}.

The main achievement of the present work is a recursive formula for cumulants (see (\ref{belleformule})), based on a repeated use of integration by parts. Note that cumulants of order greater than two {\sl are not linear operators} (for instance, the second cumulant coincides with the variance): however, our formula (\ref{belleformule}) implies that cumulants of regular functionals of Gaussian fields can be always represented as the mathematical expectation of some recursively defined random variable. We shall prove in Section \ref{S : CumCHaos} that this implies a new compact representation for cumulants of random variables belonging to a fixed Wiener chaos. We claim that this result may replace the classic ``diagram computations'' adopted in most of the probabilistic literature (see e.g. \cite{BrMa, ChaSlud, GiSu}, as well as \cite{PecTaq_SURV} for a general discussion of related combinatorial results).

\smallskip

The paper is organized as follows. In Section \ref{S : MomExp} we state and prove some useful recursive formulae for moments. Section \ref{S : Malliavin} contains basic concepts and results related to Malliavin calculus. Section \ref{S : RecursiveCum} is devoted to our main statements about cumulants on Wiener space. Finally, in Section \ref{S : CumCHaos} we specialize our results to random variables contained in a fixed Wiener chaos.

\smallskip

From now on, every random object is defined on a common suitable probability space $(\Omega, \mathcal{F}, P)$.

\bigskip

\noindent{\bf Acknowledgements.} Part of this paper has been written while we were visiting the Department of Statistics of Purdue University, in the occasion of the workshop ``Stochastic Analysis at Purdue'', from September 28 to October 1, 2009. We heartily thank Professor Frederi G. Viens for his warm hospitality and generous support.

\section{Moment expansions}\label{S : MomExp}
The starting point of our analysis (Proposition \ref{P : recmom}) is a well-known recursive relation involving moments and cumulants. As already discussed, this result is the seed (as well as a special case) of some remarkable formulae by A.D. Barbour \cite[Lemma 1 and Corollary 1]{BarbourPtrf1986}, providing Edgeworth-type expansions for smooth functions of random variables with finite moments. Since we only need Barbour's results in the special case of polynomial transformations, and for the sake of completeness, we choose to provide a self-contained presentation in this simpler setting. See also Rotar' \cite{RotarOnBar} for further extensions of Barbour's findings.

\begin{definition}[Cumulants]\label{D : cum}{\rm Let $X$ be a real-valued random variable such that $E|X|^m<\infty$ for some integer $m\geq 1$, and define $\phi_X(t) = E(e^{itX})$, $t\in\R$, to be the characteristic function of $X$.
Then, for $j=1,...,m$, the $j$th {\sl cumulant} of $X$, denoted by $\kappa_j(X)$, is given by
\begin{equation}
\kappa_j (X) = (-i)^j \frac{d^j}{d t^j} \log \phi_X (t)|_{t=0}.
\end{equation}
For instance, $\kappa_1(X) = E(X)$, $\kappa_2(X) = E(X^2)-E(X)^2 = {\rm Var}(X)$, $\kappa_3(X) = E(X^3)-3E(X^2)E(X)+2E(X)^3$, and so on.
}
\end{definition}

\smallskip

The following relation is exploited throughout the paper.

\begin{prop}\label{P : recmom}
Fix $m= 0,1,2...$, and suppose that $E|X|^{m+1}<\infty$. Then
\begin{equation}\label{EQ : RecMom}
E(X^{m+1}) = \sum_{s=0}^m \binom{m}{s}\kappa_{s+1}(X) E(X^{m-s}).
\end{equation}
\end{prop}
\begin{proof}
{} By Leibniz rule, one has that
\begin{eqnarray*}
E(X^{m+1}) &=& (-i)^{m+1} \frac{d^{m+1}}{dt^{m+1}}\phi_X(t)|_{t=0} \\
    &=& (-i)^{m+1} \left.\frac{d^{m}}{dt^{m}}\left[\left( \frac{d}{dt}\log\phi_X(t)\right)\phi_X(t)\right] \right|_{t=0} \\
    &=& \left.\left[ \sum_{s=0}^m (-i)^{s+1}\binom{m}{s}\frac{d^{s+1}}{dt^{s+1}}(\log\phi_X(t)) \times (-i)^{m-s}\frac{d^{m-s} }{dt^{m-s}}\phi_X(t)\right] \right|_{t=0},
\end{eqnarray*}
with $\frac{d^0}{dt^0}$ equal to the identity operator. This yields the desired conclusion.
\end{proof}

\smallskip

Finally, observe that (\ref{EQ : RecMom}) can be rewritten as
\[
E(X^{m+1}) = \sum_{s=0}^m \frac{\kappa_{s+1}(X)}{s!} m(m-1)\cdot\cdot\cdot(m-s+1) E(X^{m-s}),
\]
implying (by linearity) that, for $X$ as in Proposition \ref{P : recmom} and for every polynomial $f : \R\rightarrow \R$ of degree at most $m\geq 1$,
\[
E(X f(X)) = \sum_{s=0}^m \frac{\kappa_{s+1}(X)}{s!} E\left( \frac{d^s}{dx^s}f(X) \right) = \sum_{s=0}^\infty \frac{\kappa_{s+1}(X)}{s!} E\left( \frac{d^s}{dx^s}f(X) \right).
\]

\begin{rem}{\rm
In \cite[Corollary 1]{BarbourPtrf1986}, one can find sufficient conditions ensuring that the infinite expansion
\[
E(Xf(X)) = \sum_{s=0}^\infty \frac{\kappa_{s+1}(X)}{s!} E\left( \frac{d^s}{dx^s}f(X) \right).
\]
holds for some infinitely differentiable function $f$ which is not necessarily a polynomial.}
\end{rem}

\section{Malliavin operators and Gaussian analysis}\label{S : Malliavin}
We shall now present the basic elements of Gaussian analysis and Malliavin calculus that are used in this paper. The reader
is referred to \cite{Janson, MallBook, nualartbook} for any unexplained definition or
 result.

\smallskip

Let $\EuFrak H$ be a real separable Hilbert space. For any $q\geq 1$, let $\EuFrak H^{\otimes q}$ be the $q$th tensor power of $\EuFrak H$ and denote
by $\EuFrak H^{\odot q}$ the associated $q$th symmetric tensor power. We write $X=\{X(h),h\in \EuFrak H\}$ to indicate
an {\sl isonormal Gaussian process} over
$\EuFrak H$, defined on some probability space $(\Omega ,\mathcal{F},P)$.
This means that $X$ is a centered Gaussian family, whose covariance is given by the relation
$E\left[ X(h)X(g)\right] =\langle h,g\rangle _{\EuFrak H}$. We also assume that $\mathcal{F}=\sigma(X)$, that is,
$\mathcal{F}$ is generated by $X$.

For every $q\geq 1$, let $\mathcal{H}_{q}$ be the $q$th Wiener chaos of $X$,
defined as the closed linear subspace of $L^2(\Omega ,\mathcal{F},P)$
generated by the family $\{H_{q}(X(h)),h\in \EuFrak H,\left\|
h\right\| _{\EuFrak H}=1\}$, where $H_{q}$ is the $q$th Hermite polynomial
given by \[
H_q(x) = (-1)^q e^{\frac{x^2}{2}}
 \frac{d^q}{dx^q} \big( e^{-\frac{x^2}{2}} \big).
\]
We write by convention $\mathcal{H}_{0} = \mathbb{R}$. For
any $q\geq 1$, the mapping $I_{q}(h^{\otimes q})=q!H_{q}(X(h))$ can be extended to a
linear isometry between the symmetric tensor product $\EuFrak H^{\odot q}$
(equipped with the modified norm $\sqrt{q!}\left\| \cdot \right\| _{\EuFrak H^{\otimes q}}$)
and the $q$th Wiener chaos $\mathcal{H}_{q}$. For $q=0$, we write $I_{0}(c)=c$, $c\in\mathbb{R}$. It is
well-known that $L^2(\Omega):=L^2(\Omega ,\mathcal{F},P)$
can be decomposed into the infinite orthogonal sum of the spaces $\mathcal{H}_{q}$. It follows that any square integrable random variable
$F\in L^2(\Omega)$ admits the following (Wiener-It\^{o}) chaotic expansion
\begin{equation}
F=\sum_{q=0}^{\infty }I_{q}(f_{q}),  \label{E}
\end{equation}
where $f_{0}=E[F]$, and the $f_{q}\in \EuFrak H^{\odot q}$, $q\geq 1$, are
uniquely determined by $F$. For every $q\geq 0$, we denote by $J_{q}$ the
orthogonal projection operator on the $q$th Wiener chaos. In particular, if
$F\in L^2(\Omega)$ is as in (\ref{E}), then
$J_{q}F=I_{q}(f_{q})$ for every $q\geq 0$.

Let $\{e_{k},\,k\geq 1\}$ be a complete orthonormal system in $\EuFrak H$.
Given $f\in \EuFrak H^{\odot p}$ and $g\in \EuFrak H^{\odot q}$, for every
$r=0,\ldots ,p\wedge q$, the \textit{contraction} of $f$ and $g$ of order $r$
is the element of $\EuFrak H^{\otimes (p+q-2r)}$ defined by
\begin{equation}
f\otimes _{r}g=\sum_{i_{1},\ldots ,i_{r}=1}^{\infty }\langle
f,e_{i_{1}}\otimes \ldots \otimes e_{i_{r}}\rangle _{\EuFrak H^{\otimes
r}}\otimes \langle g,e_{i_{1}}\otimes \ldots \otimes e_{i_{r}}
\rangle_{\EuFrak H^{\otimes r}}.  \label{v2}
\end{equation}
Notice that the definition of $f\otimes_r g$ does not depend
on the particular choice of $\{e_k,\,k\geq 1\}$, and that
$f\otimes _{r}g$ is not necessarily symmetric; we denote its
symmetrization by $f\widetilde{\otimes }_{r}g\in \EuFrak H^{\odot (p+q-2r)}$.
Moreover, $f\otimes _{0}g=f\otimes g$ equals the tensor product of $f$ and
$g$ while, for $p=q$, $f\otimes _{q}g=\langle f,g\rangle _{\EuFrak H^{\otimes q}}$.

It can also be shown that the following {\sl multiplication formula} holds: if $f\in \EuFrak
H^{\odot p}$ and $g\in \EuFrak
H^{\odot q}$, then
\begin{eqnarray}\label{multiplication}
I_p(f) I_q(g) = \sum_{r=0}^{p \wedge q} r! {p \choose r}{ q \choose r} I_{p+q-2r} (f\widetilde{\otimes}_{r}g).
\end{eqnarray}
\smallskip

We now introduce some basic elements of the Malliavin calculus with respect
to the isonormal Gaussian process $X$. Let $\mathcal{S}$
be the set of all
cylindrical random variables of
the form
\begin{equation}
F=g\left( X(\phi _{1}),\ldots ,X(\phi _{n})\right) ,  \label{v3}
\end{equation}
where $n\geq 1$, $g:\mathbb{R}^{n}\rightarrow \mathbb{R}$ is an infinitely
differentiable function such that its partial derivatives have polynomial growth, and $\phi _{i}\in \EuFrak H$,
$i=1,\ldots,n$.
The {\sl Malliavin derivative}  of $F$ with respect to $X$ is the element of
$L^2(\Omega ,\EuFrak H)$ defined as
\begin{equation*}
DF\;=\;\sum_{i=1}^{n}\frac{\partial g}{\partial x_{i}}\left( X(\phi
_{1}),\ldots ,X(\phi _{n})\right) \phi _{i}.
\end{equation*}
In particular, $DX(h)=h$ for every $h\in \EuFrak H$. By iteration, one can
define the $m$th derivative $D^{m}F$, which is an element of $L^2(\Omega ,\EuFrak H^{\odot m})$,
for every $m\geq 2$.
For $m\geq 1$ and $p\geq 1$, ${\mathbb{D}}^{m,p}$ denotes the closure of
$\mathcal{S}$ with respect to the norm $\Vert \cdot \Vert _{m,p}$, defined by
the relation
\begin{equation*}
\Vert F\Vert _{m,p}^{p}\;=\;E\left[ |F|^{p}\right] +\sum_{i=1}^{m}E\left(
\Vert D^{i}F\Vert _{\EuFrak H^{\otimes i}}^{p}\right) .
\end{equation*}
One also writes $\mathbb{D}^{\infty} = \bigcap_{m\geq 1}
\bigcap_{p\geq 1}\mathbb{D}^{m,p}$. The Malliavin derivative $D$ obeys the following \textsl{chain rule}. If
$\varphi :\mathbb{R}^{n}\rightarrow \mathbb{R}$ is continuously
differentiable with bounded partial derivatives and if $F=(F_{1},\ldots
,F_{n})$ is a vector of elements of ${\mathbb{D}}^{1,2}$, then $\varphi
(F)\in {\mathbb{D}}^{1,2}$ and
\begin{equation*}
D\,\varphi (F)=\sum_{i=1}^{n}\frac{\partial \varphi }{\partial x_{i}}(F)DF_{i}.
\end{equation*}
Note also that a random variable $F$ as in (\ref{E}) is in ${\mathbb{D}}^{1,2}$ if and only if
$\sum_{q=1}^{\infty }q\|J_qF\|^2_{L^2(\Omega)}<\infty$
and, in this case, $E\left( \Vert DF\Vert _{\EuFrak H}^{2}\right)
=\sum_{q=1}^{\infty }q\|J_qF\|^2_{L^2(\Omega)}$. If $\EuFrak H=
L^{2}(\mathbb{A},\mathscr{A},\mu )$ (with $\mu $ non-atomic), then the
derivative of a random variable $F$ as in (\ref{E}) can be identified with
the element of $L^2(\mathbb{A}\times \Omega )$ given by
\begin{equation}
D_{x}F=\sum_{q=1}^{\infty }qI_{q-1}\left( f_{q}(\cdot ,x)\right) ,\quad x\in
\mathbb{A}.  \label{dtf}
\end{equation}

We denote by $\delta $ the adjoint of the operator $D$, also called the
\textsl{divergence operator}. A random element $u\in L^2(\Omega ,\EuFrak H)$
belongs to the domain of $\delta $, noted $\mathrm{Dom}\delta $, if and
only if it verifies
$|E\langle DF,u\rangle _{\EuFrak H}|\leq c_{u}\,\Vert F\Vert _{L^2(\Omega)}$
for any $F\in \mathbb{D}^{1,2}$, where $c_{u}$ is a constant depending only
on $u$. If $u\in \mathrm{Dom}\delta $, then the random variable $\delta (u)$
is defined by the duality relationship (called \textsl{integration by parts
formula})
\begin{equation}
E(F\delta (u))=E\langle DF,u\rangle _{\EuFrak H},  \label{ipp}
\end{equation}
which holds for every $F\in {\mathbb{D}}^{1,2}$.

The family $(P_t,\,t\geq 0)$ of operators is defined through the projection operators $J_q$
as \begin{equation}\label{OUsemigroup}
P_t=\sum_{q=0}^\infty e^{-qt}J_q,
\end{equation}
and is called the {\sl Ornstein-Uhlenbeck semigroup}. Assume that
the process $X'$, which stands for an independent copy of $X$, is such that $X$ and $X'$ are defined
on the product probability space $(\Omega\times\Omega',\mathscr{F}\otimes\mathscr{F}',P\times P')$.
Given a random variable $F\in\mathbb{D}^{1,2}$, we can regard it
as a measurable mapping from
$\R^\HH$ to $\R$, determined $P\circ X^{-1}$-almost surely. Then, for any $t\geq 0$, we have the
so-called {\sl Mehler's formula}:
\begin{equation}\label{mehler}
P_tF=E'\big(F(e^{-t}X+\sqrt{1-e^{-2t}}X')\big),
\end{equation}
where $E'$ denotes the mathematical expectation with respect to the probability $P'$.
By means of this formula, it is immediate to prove that $P_t$ is a contraction
operator on $L^p(\Omega)$, for all $p\geq 1$.

The operator $L$ is defined as
$L=\sum_{q=0}^{\infty }-qJ_{q}$,
and it can be proven to be the infinitesimal generator of the Ornstein-Uhlenbeck
semigroup $(P_t)_{t\geq 0}$.
The domain of $L$ is
\begin{equation*}
\mathrm{Dom}L=\{F\in L^2(\Omega ):\sum_{q=1}^{\infty }q^{2}\left\|
J_{q}F\right\| _{L^2(\Omega )}^{2}<\infty \}=\mathbb{D}^{2,2}\text{.}
\end{equation*}
There is an important relation between the operators $D$, $\delta $ and $L$.
A random variable $F$ belongs to
$\mathbb{D}^{2,2}$ if and only if $F\in \mathrm{Dom}\left( \delta D\right) $
(i.e. $F\in {\mathbb{D}}^{1,2}$ and $DF\in \mathrm{Dom}\delta $) and, in
this case,
\begin{equation}
\delta DF=-LF.  \label{k1}
\end{equation}

For any $F \in L^2(\Omega )$, we define $L^{-1}F =\sum_{q=0}^{\infty }-\frac{1}{q} J_{q}(F)$. The operator $L^{-1}$ is called the
\textsl{pseudo-inverse} of $L$. Indeed, for any $F \in L^2(\Omega )$, we have that $L^{-1} F \in  \mathrm{Dom}L
= \mathbb{D}^{2,2}$,
and
\begin{equation}\label{Lmoins1}
LL^{-1} F = F - E(F).
\end{equation}

We now present two useful lemmas, that we will need throughout the sequel. The first statement exploits
the two fundamental relations (\ref{k1}) and (\ref{Lmoins1}). Note that these relations have been extensively applied in \cite{NP-PTRF}, in the context of the normal approximation of functionals of Gaussian fields by means of Stein's method.

\begin{lemma}\label{L : Tech1}
Suppose that $F\in\mathbb{D}^{1,2}$ and $G\in L^2(\Omega)$. Then,
$L^{-1}G \in \mathbb{D}^{2,2}$ and we have:
\begin{equation}
E(FG) = E(F)E(G)+E(\langle DF,-DL^{-1}G\rangle_{\HH}).
\end{equation}
\end{lemma}
\begin{proof}
{}By (\ref{k1}) and (\ref{Lmoins1}),
\[
E(FG)-E(F)E(G)=E(F(G-E(G)))=E(F\times LL^{-1}G)=E(F \delta(- DL^{-1}G) ),
\]
and the result is obtained by using the integration by parts formula (\ref{ipp}).
\end{proof}

\begin{rem}{\rm
Observe that $\langle DF,-DL^{-1}G\rangle_{\HH} $ is not necessarily square-integrable, albeit it is by
construction in $L^1(\Omega)$. On the other hand, we have
\begin{equation}\label{EQ : Ptcum}
\langle DF,-DL^{-1}G\rangle_{\HH} = \int_0^{\infty}e^{-t} \langle DF,P_t DG\rangle_{\HH}\,dt,
\end{equation}
see indeed identity (3.46) in \cite{NP-PTRF}.
}
\end{rem}

The next result is a consequence of the previous Lemma \ref{L : Tech1}.

\begin{prop}
Fix $p\geq 2$, and assume that $F\in L^{4p-4}(\Omega) \cap \mathbb{D}^{1,4}$.
Then, $F^p\in \mathbb{D}^{1,2}$ and  $DF^p = pF^{p-1}DF$. Moreover, for any $G\in L^2(\Omega)$,
\begin{equation}\label{EQ : simple}
E(F^pG) = E(F^p)E(G)+pE(F^{p-1}\langle DF,-DL^{-1}G\rangle_{\HH}).
\end{equation}
\end{prop}

\section{A recursive formula for cumulants}\label{S : RecursiveCum}

The aim of this section is to deduce from formula (\ref{EQ : RecMom}) a recursive relation for cumulants of sufficiently regular functionals of the isonormal process $X$. To do this, we need to (recursively) introduce some further notation.

\begin{definition}\label{Def : Gamma}{\rm Let $F\in \mathbb{D}^{1,2}$. We define $\Gamma_0(F) = F$
and $\Gamma_1(F) = \langle DF,-DL^{-1}F\rangle_{\HH}$. If, for $j\geq 1$,
the random variable $\Gamma_j(F)$ is a well-defined element of $L^2(\Omega)$,
we set $\Gamma_{j+1}(F) = \langle DF,-DL^{-1}\Gamma_j(F)\rangle_{\HH}$.
Observe that the definition of $\Gamma_{j+1}(F)$ is well given, since (as already observed in general) the square-integrability
of $\Gamma_j(F)$ implies that $L^{-1}\Gamma_j(F) \in {\rm Dom}\,L=\mathbb{D}^{2,2}\subset \mathbb{D}^{1,2}$.
}
\end{definition}

The following statement provides sufficient conditions on $F$, ensuring that the random variable
$\Gamma_j(F)$ is well defined.

\begin{lemma}\label{easyrem}
\begin{enumerate}
\item Fix an integers $j\geq 1$, and let $F,G\in\mathbb{D}^{j,2^j}$. Then
$\langle DF,-DL^{-1}G\rangle_\HH\in \mathbb{D}^{j-1,2^{j-1}}$,
where we set by convention (but consistently!) $\mathbb{D}^{0,1}=L^1(\Omega)$.
\item Fix an integer $j\geq 1$, and let $F\in\mathbb{D}^{j,2^j}$. Then, for all
$k=1,\ldots,j$, we have that $\Gamma_k(F)$ is a well-defined element of $\mathbb{D}^{j-k,2^{j-k}}$;
in particular, one has that $\Gamma_k(F)\in L^1(\Omega)$ and that the quantity
$E[\Gamma_k(F)]$ is well-defined and finite.
\item If $F\in\mathbb{D}^{\infty}$ (in particular if $F$ equals a {\rm finite} sum of multiple integrals), then $\Gamma_j(F) \in \mathbb{D}^{\infty}$ for every $j\geq 0$.
\end{enumerate}
\end{lemma}
{\it Proof.}
Without loss of generality, we assume during all the proof that $\HH$ has the form
$L^{2}(\mathbb{A},\mathscr{A},\mu)$, where $(\mathbb{A},\mathscr{A})$ is
a measurable space and $\mu$ is a $\sigma$-finite measure with no atoms.

1. Let $k\in\{0,\ldots,j-1\}$. Using Leibniz rule  for $D$
(see e.g. \cite[Exercice 1.2.13]{nualartbook}), one has that
\begin{eqnarray} \label{Lauma}
-D^k
\langle DF,-DL^{-1}G\rangle_\HH&=&
D^k
\int_{\mathbb{A}}D_aF\,D_aL^{-1}G\,\mu(da)\\
&=&\sum_{l=0}^k \binom{k}{l}
\int_{\mathbb{A}}D^{k-l}(D_aF)\widetilde{\otimes} D^{l}(D_aL^{-1}G)\,\mu(da), \notag
\end{eqnarray}
with $\widetilde{\otimes}$ the usual symmetric tensor product.
%Note that, to deduce (\ref{Lauma}), it is sufficient to consider random variables
%$F,G$ that are equal to a finite sum of multiple integrals, and then to apply
%a standard approximation argument (for instance, by using the fact that such random
%variables are dense in each space $\mathbb{D}^{a,b}$, $a\geq 1$, $b>1$ --
%see \cite[Corollary 1.5.1]{nualartbook}).
Note that, to deduce (\ref{Lauma}), it is sufficient to consider random variables $F,G$ that have the form (\ref{v3}) (for which the formula is evident, since it basically boils down to the original Leibniz rule for differential calculus), and then to apply a standard approximation argument.
From (\ref{Lauma}), one therefore deduces that
\begin{eqnarray}
&&\|D^k
\langle DF,-DL^{-1}G\rangle_\HH\|_{\HH^{\otimes k}}\notag\\
&\leq&\sum_{l=0}^k \binom{k}{l}
\left\|\int_{\mathbb{A}}D^{k-l}(D_aF)\widetilde{\otimes} D^{l}(D_aL^{-1}G)\,\mu(da)\right\|_{\HH^{\otimes k}}\notag\\
&\leq&\sum_{l=0}^k \binom{k}{l}
\int_{\mathbb{A}}\|D^{k-l}(D_aF)\|_{\HH^{\otimes(k-l)}}\| D^{l}(D_aL^{-1}G)\|_{\HH^{\otimes l}}\mu(da)\notag\\
&\leq&\sum_{l=0}^k \binom{k}{l}
\sqrt{\int_{\mathbb{A}}\|D^{k-l}(D_aF)\|_{\HH^{\otimes(k-l)}}^2 \mu(da)}
\sqrt{\int_{\mathbb{A}}\|D^{l}(D_aL^{-1}G)\|^2_{\HH^{\otimes l}}\mu(da)}\notag\\
&=&\sum_{l=0}^k \binom{k}{l}
\|D^{k-l+1}F\|_{\HH^{\otimes(k-l+1)}}
\|D^{l+1}L^{-1}G\|_{\HH^{\otimes (l+1)}}.\label{coujou}
\end{eqnarray}
By mimicking the arguments used in the proof of Proposition 3.1 in \cite{Noupecrei2}
(see also (\ref{EQ : Ptcum})), it is possible
to prove that
\[
-D^{l+1}L^{-1}G =\int_0^\infty e^{-(l+1)t}P_tD^{l+1}G dt.
\]
Consequently, for any real $p\geq 1$,
\begin{eqnarray}
E\big[\|D^{l+1}L^{-1}G\|^p_{\HH^{\otimes (l+1)}}\big]
&\leq&
E\left[\left(\int_0^\infty e^{-(l+1)t}\|P_tD^{l+1}G\|_{\HH^{\otimes (l+1)}} dt\right)^p\right]\notag\\
&\leq&
\frac{1}{(l+1)^{p-1}}
\int_0^\infty e^{-(l+1)t}E\big[\|P_tD^{l+1}G\|^p_{\HH^{\otimes (l+1)}}\big]dt\notag\\
&\leq&
\frac{1}{(l+1)^{p-1}}E\big[\|D^{l+1}G\|^p_{\HH^{\otimes (l+1)}}\big]
\int_0^\infty e^{-(l+1)t}dt\notag\\
&=&\frac{1}{(l+1)^p}\,E\big[\|D^{l+1}G\|^p_{\HH^{\otimes (l+1)}}\big],\label{coujou2}
\end{eqnarray}
where, to get the last inequality, we used the contraction property of $P_t$ on $L^p(\Omega)$.
Finally, by combining (\ref{coujou}) with (\ref{coujou2}) through the Cauchy-Schwarz inequality
on the one hand,
and by the assumptions on $F$ and $G$ on the other hand, we immediately get
that $\|D^k\langle DF,-DL^{-1}G\rangle_\HH\|_{\HH^{\otimes k}}$ belongs to
$L^{2^{j-1}}(\Omega)$ for
all $k=0,\ldots, j-1$, yielding the announced result.

2. Fix $j\geq 1$ and $F\in\mathbb{D}^{j,2^j}$.
The proof is achieved by recursion on $k$.
For $k=1$, the desired property is true, due to Point 1 applied to $G=F$.
Now, assume that the desired property is true for $k$ ($\leq j-1$), and let us prove that it also holds for $k+1$.
Indeed, we have that $\Gamma_{k+1}(F)=\langle DF,-DL^{-1}\Gamma_k(F)\rangle_\HH$,
that $F\in\mathbb{D}^{j,2^j}\subset \mathbb{D}^{j-k,2^{j-k}}$
(assumption on $F$) and that
$\Gamma_k(F)\in\mathbb{D}^{j-k,2^{j-k}}$ (recurrence assumption). Point 1
yields the desired conclusion.

3. The proof is immediately obtained by a repeated application of Point 2.
\fin

We will now provide a new characterization of cumulants for functionals of Gaussian processes: it is the
fundamental tool yielding the main results of the paper. Note that, due to
Lemma \ref{easyrem} (Point 2), the following statement applies
in particular to random variables in $\mathbb{D}^{m,2^m}$ ($m\geq 2$).

\begin{theorem}\label{T : MAIN} Fix an integer $m\geq 2$, and suppose that:
(i) the random variable $F$ is an element of $L^{4m-4}(\Omega)\cap \mathbb{D}^{1,4}$,
(ii) for every $j\leq m-1$, the random variable $\Gamma_{j}(F)$ is in $L^2(\Omega)$.
Then, for every $s\leq m$,
\begin{equation}\label{belleformule}
\kappa_{s+1}(F) = s!E[\Gamma_s(F)].
\end{equation}
\end{theorem}
\begin{proof}{}
The proof is achieved by recursion on $s$. First observe that $\kappa_1(F) = E(F) =
 E[\Gamma_0(F)]$, so that the claim is proved for $s=0$. Now suppose that (\ref{belleformule}) holds for every $s=0,...,l$, where $l\leq m-1$. According to (\ref{EQ : RecMom}), we have that
\begin{equation}\label{jimmy}
E(F^{l+1}) = \sum_{s=0}^{l-1}\binom{l}{s} \kappa_{s+1}(F) E(F^{l-s}) + \kappa_{l+1}(F).
\end{equation}
On the other hand, by applying (\ref{EQ : simple}) to the case $p=l$ and $G=F$, we deduce that
\[
E(F^{l+1})=E(F^{l})E(F)+  lE(F^{l-1}\langle DF,-DL^{-1}F\rangle_{\HH}) = E(F^{l})\kappa_1(F)+  lE(F^{l-1}\Gamma_1(F)).
\]
By the recurrence assumption, and by applying again (\ref{EQ : simple}) to the case $p=l-1$ and $G=\Gamma_1(F)$, one deduces therefore that
\begin{eqnarray*}
&& E(F^{l+1})=E(F^{l})\kappa_1(F)+  lE(F^{l-1}\Gamma_1(F))\\
&&= \kappa_1(F)E(F^{l})+ \kappa_2(F) l E(F^{l-1}) + l(l-1)E(F^{l-2}\Gamma_2(F)) .
\end{eqnarray*}
Iterating this procedure yields eventually
\[
E(F^{l+1}) = \sum_{s=0}^{l-1}\binom{l}{s} \kappa_{s+1}(F) E(F^{l-s}) + l! E[\Gamma_l(F)],
\]
so that one deduces from (\ref{jimmy}) that relation (\ref{belleformule}) holds for $s=l+1$.
The proof is concluded.
\end{proof}

\begin{rem}{\rm Suppose that $F=I_q(f)$, where $q\geq 2$ and $f\in \HH^{\odot q}$. Then, $L^{-1}F = -q^{-1}F$, and consequently
\begin{eqnarray}\label{multCUM}
E[\Gamma_s(F)] &=& E[\langle DF , -DL^{-1}\Gamma_{s-1}(F)\rangle_\HH]
=E[F\Gamma_{s-1}(F)] \\
&=& E[\langle - DL^{-1}F , D\Gamma_{s-1}(F)\rangle_\HH] =\frac1q E[\langle DF , D\Gamma_{s-1}(F)\rangle_\HH].\notag
\end{eqnarray}
}
\end{rem}

Several applications of formula (\ref{belleformule}) (and, implicitly, of (\ref{multCUM})) are detailed in the next section.

\section{Cumulants on Wiener chaos}\label{S : CumCHaos}
\subsection{General statement}
We now focus on the computation of cumulants associated to random variables belonging to a fixed Wiener chaos,
that is, random variables having the form of a multiple Wiener-It\^{o} integral. Our main findings are
collected in the following statement, providing a quite compact representation for cumulants associated with
multiple integrals of arbitrary orders. In the forthcoming Section \ref{SS : Diagrams}, we will compare our results with the classic {\sl diagram formulae} that are customarily used in the probabilistic literature.
\begin{thm}\label{thm-pasmal}
Let $q\geq 2$, and assume that $F=I_q(f)$, where $f\in\HH^{\odot q}$. Denote by $\kappa_s(F)$,
$s\geq1$, the cumulants of $F$.
We have $\kappa_1(F)=0$, $\kappa_2(F)=q!\|f\|^2_{\HH^{\otimes q}}$ and,
for every $s\geq 3$,
\begin{equation}\label{formula-cumulants}
\kappa_s(F)= q!(s-1)!\sum c_q(r_1,\ldots,r_{s-2})
\big\langle (...((f\widetilde{\otimes}_{r_1} f) \widetilde{\otimes}_{r_2} f)\ldots
\widetilde{\otimes}_{r_{s-3}}f)\widetilde{\otimes}_{r_{s-2}}f,f\big\rangle_{\HH^{\otimes q}},
\end{equation}
where the sum $\sum$ runs over all collections of integers $r_1,\ldots,r_{s-2}$ such that:
\begin{enumerate}
\item[(i)] $1\leq r_1,\ldots, r_{s-2}\leq q$;
\item[(ii)] $r_1+\ldots+r_{s-2}=\frac{(s-2)q}{2}$;
\item[(iii)] $r_1<q$, $r_1+r_2 <\frac{3q}{2}$, $\ldots$, $r_1+\ldots +r_{s-3}< \frac{(s-2)q}2$;
\item[(iv)] $r_2\leq 2q-2r_1$, $\ldots$, $r_{s-2}\leq (s-2)q-2r_1-\ldots-2r_{s-3}$;
\end{enumerate}
and where the combinatorial constants $c_q(r_1,\ldots,r_{s-2})$ are recursively defined by the
relations
\[
c_q(r)=q(r-1)!\binom{q-1}{r-1}^2,
\]
and, for $a\geq 2$,
\[
c_q(r_1,\ldots,r_{a})=q(r_{a}-1)!\binom{aq-2r_1-\ldots - 2r_{a-1}-1}{r_{a}-1}
\binom{q-1}{r_{a}-1}c_q(r_1,\ldots,r_{a-1}).
\]
\end{thm}

\begin{remark}
{\rm
\begin{enumerate}
\item
If $sq$ is odd, then $\kappa_s(F)=0$, see indeed condition $(ii)$.
\item
If $q=2$ and $F=I_2(f)$, $f\in\HH^{\odot 2}$, then the only possible integers $r_1,\ldots,r_{s-2}$ verifying $(i)-(iv)$ in the previous statement are $r_1=\ldots=r_{s-2}=1$.
On the other hand, we immediately compute that $c_2(1)=2$, $c_2(1,1)=4$, $c_2(1,1,1)=8$, and so on.
Therefore,
\[
\kappa_s(F)= 2^{s-1}(s-1)!\big\langle (...(f\otimes_1 f)\ldots f)\otimes_1 f, f\big\rangle_{\HH^{\otimes 2}},
\]
and we recover the classical expression of the cumulants of a double integral (as used e.g. in
\cite{FoxTaq} or \cite{exact}).
\item If $q\geq 2$ and $F=I_q(f)$, $f\in\HH^{\odot q}$, then (\ref{formula-cumulants})
for $s=4$ reads
\begin{eqnarray}
\kappa_4(I_q(f))&=&6q!\sum_{r=1}^{q-1}c_q(r,q-r)\big\langle (f\widetilde{\otimes}_rf)\widetilde{\otimes}_{q-r}
f,f\big\rangle_{\HH^{\otimes q}}\notag\\
&=&\frac{3}{q}\sum_{r=1}^{q-1}rr!^2\binom{q}{r}^4(2q-2r)!\big\langle (f\widetilde{\otimes}_rf)\otimes_{q-r}
f,f\big\rangle_{\HH^{\otimes q}}\notag\\
&=&\frac{3}{q}\sum_{r=1}^{q-1}rr!^2\binom{q}{r}^4(2q-2r)!\big\langle f\widetilde{\otimes}_rf,f\otimes_r f
\big\rangle_{\HH^{\otimes (2q-2r)}}\notag\\
&=&\frac{3}{q}\sum_{r=1}^{q-1}rr!^2\binom{q}{r}^4(2q-2r)!
\| f\widetilde{\otimes}_rf\|^2_{\HH^{\otimes (2q-2r)}},
\label{bob}
\end{eqnarray}
and we recover the expression for $\kappa_4(F)$ deduced in \cite[Section 3.1]{Noupecrei3}  by a different route.
Formula (\ref{bob}) should be compared with the
following identity, first established in \cite[p. 183]{nunugio}:
for every $q\geq 2$ and every $f \in \HH^{\odot q}$,
\begin{eqnarray}
\kappa_4(I_q(f)) &=&
\sum_{r=1}^{q-1}\frac{q!^{4}}{r!^2\left( q-r\right) !^{2}}\Big[ \left\| f\otimes
_{r}f\right\| _{\HH^{\otimes (2q-2r) }}^{2}
\label{SelfCont} \\
&&\text{ \ \ \ \ \ \ \ \ \ \ \ \ \ \ \ \ \ \ \ \ \ \ \ \ \ \ \ \ \
\ \ \ \ \ \ \ }+\left. \binom{2q-2r}{q-r}\left\|
f\widetilde{\otimes} _{r}f\right\| _{\HH^{\otimes
(2q-2r) }}^{2}\right] . \notag
\end{eqnarray}
Note that it is in principle much more difficult to deal with (\ref{SelfCont}), since it involves both symmetrized and non-symmetrized contractions.
\end{enumerate}
}
\end{remark}
{\it Proof of Theorem \ref{thm-pasmal}}.
Let us first show that the following formula (\ref{for})
is in order: for any $s\geq 2$,
\begin{eqnarray}
\Gamma_{s-1}(F)&=&
\sum_{r_1=1}^{q} \ldots\sum_{r_{s-1}=1}^{[(s-1)q-2r_1-\ldots-2r_{s-2}]\wedge q}c_q(r_1,\ldots,r_{s-1})
{\bf 1}_{\{r_1< q\}}
\ldots {\bf 1}_{\{
r_1+\ldots+r_{s-2}< \frac{(s-1)q}2
\}}\notag\\
&&
\hskip5cm\times I_{sq-2r_1-\ldots-2r_{s-1}}\big(
(...(f\widetilde{\otimes}_{r_1} f) \widetilde{\otimes}_{r_2} f)\ldots
f)\widetilde{\otimes}_{r_{s-1}}f
\big).\notag\\
\label{for}
\end{eqnarray}
We shall prove (\ref{for}) by induction, assuming without loss of generality that $\HH$ has the form
$L^{2}(\mathbb{A},\mathscr{A},\mu)$, where $(\mathbb{A},\mathscr{A})$ is
a measurable space and $\mu$ is a $\sigma$-finite measure without atoms.
When $s=2$, identity (\ref{for}) simply reads
\begin{equation}\label{gamma1}
\Gamma_{1}(F)=
\sum_{r=1}^{q} c_q(r)
I_{2q-2r}(
f\widetilde{\otimes}_{r} f).
\end{equation}
Let us prove (\ref{gamma1}) by means of the multiplication formula
(\ref{multiplication}) , see also \cite{NO} for similar computations.
We have
\begin{eqnarray*}
\Gamma_1(F)&=&\langle DF,-DL^{-1}F\rangle_\HH = \frac1q\|DF\|^2_\HH=q\int_\mathbb{A}
I_{q-1}\big(f(\cdot,a)\big)^2\mu(da)\\
&=&q\sum_{r=0}^{q-1}r!\binom{q-1}{r}^2 I_{2q-2-2r}\left(\int_\mathbb{A}f(\cdot,a)\widetilde{\otimes}_r f(\cdot,a)
\mu(da)\right)\\
&=&q\sum_{r=0}^{q-1}r!\binom{q-1}{r}^2 I_{2q-2-2r}\left(f\widetilde{\otimes}_{r+1} f\right)\\
&=&q\sum_{r=1}^{q}(r-1)!\binom{q-1}{r-1}^2 I_{2q-2r}\left(f\widetilde{\otimes}_{r} f\right),
\end{eqnarray*}
thus yielding (\ref{gamma1}).
Assume now that (\ref{for}) holds for $\Gamma_{s-1}(F)$, and let us prove that it continues to hold for $\Gamma_s(F)$.
We have, still using the multiplication formula (\ref{multiplication}) and proceeding
as above,
\begin{eqnarray*}
\Gamma_{s}(F)&=&\langle DF,-DL^{-1}\Gamma_{s-1}F\rangle_\HH\\
&=&\sum_{r_1=1}^{q} \ldots\sum_{r_{s-1}=1}^{[(s-1)q-2r_1-\ldots-2r_{s-2}]\wedge q}q
c_q(r_1,\ldots,r_{s-1})
{\bf 1}_{\{r_1< q\}}
\ldots {\bf 1}_{\{
r_1+\ldots+r_{s-2}< \frac{(s-1)q}2
\}}\notag\\
&&
\hskip1cm\times
{\bf 1}_{\{
r_1+\ldots+r_{s-1}< \frac{sq}2
\}} \big\langle I_{q-1}(f),I_{sq-2r_1-\ldots-2r_{s-1}-1}\big(
(...(f\widetilde{\otimes}_{r_1} f) \widetilde{\otimes}_{r_2} f)\ldots
f)\widetilde{\otimes}_{r_{s-1}}f
\big)\big\rangle_\HH\\
&=&\sum_{r_1=1}^{q} \ldots\sum_{r_{s-1}=1}^{[(s-1)q-2r_1-\ldots-2r_{s-2}]\wedge q}\,\,\,
\sum_{r_{s}=1}^{[sq-2r_1-\ldots-2r_{s-1}]\wedge q}
c_q(r_1,\ldots,r_{s-1})\times q
(r_s-1)!\\
&&\hskip1cm\times\binom{sq-2r_1-\ldots-2r_{s-1}-1}{r_s-1}\binom{q-1}{r_s-1}
{\bf 1}_{\{r_1< q\}}
\ldots {\bf 1}_{\{
r_1+\ldots+r_{s-2}< \frac{(s-1)q}2
\}}\notag\\
&&
\hskip1cm\times
{\bf 1}_{\{
r_1+\ldots+r_{s-1}< \frac{sq}2
\}} I_{(s+1)q-2r_1-\ldots-2r_{s}}\big(
(...(f\widetilde{\otimes}_{r_1} f) \widetilde{\otimes}_{r_2} f)\ldots
f)\widetilde{\otimes}_{r_{s}}f
\big),
\end{eqnarray*}
which is the desired formula for $\Gamma_s(F)$.
The proof of (\ref{for}) for all $s\geq 1$ is thus finished.
Now, let us take the expectation on both sides of (\ref{for}). We get
\begin{eqnarray*}
\kappa_{s}(F)&=&(s-1)!E[\Gamma_{s-1}(F)]\\
&=&(s-1)!
\sum_{r_1=1}^{q} \ldots\sum_{r_{s-1}=1}^{[(s-1)q-2r_1-\ldots-2r_{s-2}]\wedge q}c_q(r_1,\ldots,r_{s-1})
{\bf 1}_{\{r_1< q\}}
\ldots {\bf 1}_{\{
r_1+\ldots+r_{s-2}< \frac{(s-1)q}2
\}}
\\
&&
\hskip4cm\times {\bf 1}_{\{
r_1+\ldots+r_{s-1}= \frac{sq}2
\}}\times
(...(f\widetilde{\otimes}_{r_1} f) \widetilde{\otimes}_{r_2} f)\ldots
f)\widetilde{\otimes}_{r_{s-1}}f.
\end{eqnarray*}
Observe that, if $2r_1+\ldots+2r_{s-1}= sq$ and $r_{s-1}\leq
(s-1)q-2r_1-\ldots-2r_{s-2}$ then
$2r_{s-1}=q+(s-1)q-2r_1-\ldots-2r_{s-2}\geq q+r_{s-1}$, so that $r_{s-1}\geq q$.
Therefore,
\begin{eqnarray*}
\kappa_{s}(F)&=&(s-1)!
\sum_{r_1=1}^{q} \ldots\sum_{r_{s-2}=1}^{[(s-2)q-2r_1-\ldots-2r_{s-3}]\wedge q}c_q(r_1,\ldots,r_{s-2},q)
{\bf 1}_{\{r_1< q\}}
\ldots {\bf 1}_{\{
r_1+\ldots+r_{s-3}< \frac{(s-2)q}2
\}}
\\
&&
\hskip4cm\times {\bf 1}_{\{
r_1+\ldots+r_{s-2}= \frac{(s-2)q}2
\}}
\big\langle(...( f\widetilde{\otimes}_{r_1} f) \widetilde{\otimes}_{r_2} f)\ldots
f)\widetilde{\otimes}_{r_{s-2}}f,f\big\rangle_{\HH^{\otimes q}},
\end{eqnarray*}
which is the announced result,
since $c_q(r_1,\ldots,r_{s-2},q)=q!c_q(r_1,\ldots,r_{s-2})$.
\fin

\subsection{Combinatorial expression of cumulants}\label{SS : Diagrams}
We now provide a classic combinatorial representation of cumulants of the type $\kappa_s(F)$, in the case
where: (i) $s\geq 2$, (ii) $q\geq 2$, (iii) $sq$ is even,
(iv) $F=I_q(f)$ (with $f\in\HH^{\odot q}$) and (v) $\HH = L^2(\mathbb{A},\mathscr{A},\mu)$ is a non-atomic measure space. Assumptions (i)-(v) will be in order throughout this section. As explained e.g. in \cite{PecTaq_SURV}, one can equivalently express cumulants of chaotic random variables by using {\sl diagrams} or {\sl graphs}: here, we choose to adopt the (somewhat simpler) representation in terms of graphs. See \cite[Section 4]{PecTaq_SURV} for an explicit connection between graphs and cumulants; see \cite{MarPTRF} for some striking application of graph counting to the computation of cumulants of non-linear functionals of spherical Gaussian fields; see \cite{BrMa, ChaSlud, GiSu} for classic examples of how to use diagram enumeration to deduce CLTs for Gaussian subordinated fields.

\begin{definition}{\rm For $s\geq 2$, consider the set $[s] = \{1,...,s\}$ of the first $s$ integers. For
$q\geq 2$, we denote by $K(s,q)$ the class of all non-oriented graphs $\gamma$ over $[s]$ satisfying the
following properties:
\begin{itemize}
\item[-] $\gamma$ does not contain edges of the type $(j,j)$, $j=1,...,s$, that is, no edges of $\gamma$ connect a vertex with itself. One can equivalently say that $\gamma$ ``has no loops''.
\item[-] Multiple edges are allowed, that is, an edge can appear $k\geq 2$ times into $\gamma$; in this case, we say that $k$ is the {\sl multiplicity} of the edge. Also, if $i,j$ are connected by an edge of multiplicity $k$, we say that $i$ and $j$ {\sl are connected $k$ times}.
\item[-] Every vertex appears in exactly $q$ edges (counting multiplicities).
\item[-] $\gamma$ is connected.
\end{itemize}
If $sq$ is odd, then $K(s,q)$ is empty. If $sq$ is even, then each $\gamma\in K(s,q)$ contains exactly $sq/2$ edges (counting multiplicities). For instance: an element of $K(3,2)$ is $\gamma =\{(1,2),(2,3),(3,1)\}$; an element of $K(3,4)$ is $\gamma =\{(1,2),(1,2),(2,3),(2,3),(1,3),(1,3)\}$ (note that each edge has multiplicity 2).
}
\end{definition}

\begin{definition} {\rm Given $q,s\geq 2$ such that $sq$ is even, and $\gamma\in K(s,q)$, we define the constant $w(\gamma)$ as follows.
\begin{itemize}
\item[(a)] For every $j=1,...,s$, consider a generic vector $L(j) = (l(j,1),...,l(j,q))$ of $q$ distinct objects. Write $\{L(j)\}$ for the set of the components of $L(j)$.
\item[(b)] Starting from $\gamma$, build a matching $m(\gamma)$ over $L := \bigcup_{j=1,...,s} \{L(j)\} $ as follows. Enumerate the vertices $v_1,...,v_{sq/2}$ of $\gamma$. If $v_1$ links $i_1$ and $j_1$ and has multiplicity $k_1$, then pick $k_1$ elements of $L(i_1)$ and $k_1$ elements of $L(j_1)$ and build a matching between the two $k_1$-sets. If $v_2$ links $i_2$ and $j_2$ and has multiplicity $k_2$, then pick $k_2$ elements of $L(i_2)$ (among those not already chosen at the previous step, whenever $i_2$ equals $i_1$ or $j_1$) and $k_2$ elements of $L(j_2)$ (among those not already chosen at the previous step, whenever $j_2$ equals $i_1$ or $j_1$) and build a matching between the two $k_2$-sets. Repeat the operation up to the step $sq/2$.
\item[(c)] Define $\EuFrak{S}_q$ as the group of all permutations of $[q]$. For every $\sigma\in \EuFrak{S}_q$, define the vector $L_\sigma(j) = (l_\sigma(j,1 ),...,l_\sigma(j,q))$, where $l_\sigma(j,p) = l(j,\sigma(p))$, $p=1,...,q$.
\item[(d)] Define $\EuFrak{S}^{(s)}_q$ as the $s$th product group of $\EuFrak{S}_q$, that is, $\EuFrak{S}^{(s)}_q$ is the collection of all $s$-vectors of the type ${\bf \sigma} = (\sigma_1,...,\sigma_s)$, where $\sigma_j \in \EuFrak{S}_q$, $j=1,...,s$, endowed with the usual product group structure.
\item[(e)] For every ${\bf \sigma}=(\sigma_1,...,\sigma_s) \in \EuFrak{S}^{(s)}_q$,  build a new matching $m_{\bf \sigma}(\gamma)$ over $L$ by repeating the same operation as at point (b), with the vector $L_{\sigma_j}(j) $ replacing $L(j)$ for every $j=1,...,s$.
\item[(f)] Define an equivalence relation over $\EuFrak{S}^{(s)}_q$ by writing ${\bf \sigma } \sim_\gamma {\bf \pi }$ whenever $m_{\bf \sigma}(\gamma) = m_{\bf \pi}(\gamma)$. Let $\EuFrak{S}^{(s)}_q/\sim_\gamma$ be the quotient of $\EuFrak{S}^{(s)}_q$ with respect to $\sim_\gamma$.
\item[(g)] Define $w(\gamma)$ to be the cardinality of $\EuFrak{S}^{(s)}_q/\sim_\gamma$.
\end{itemize}
For instance, one can prove that $w(\gamma) = 2^{s-1}$ for every $s\geq 2$ and every $\gamma\in K(s,2)$. Also, $w(\gamma) = q!$ for every $q\geq 2$ and every $\gamma \in K(2,q)$.
}
\end{definition}

\begin{definition}{\rm For $q\geq 2$, let $f\in\HH^{\odot q}$, that is, $f$ is a symmetric element of $L^2(\mathbb{A}^q,\mathscr{A}^q,\mu^q)$. Fix $\gamma \in K(s,q)$, where $s\geq 2$ and $sq$ is even. Starting from $f$ and $\gamma$, we define a function
\[
(a_1,...,a_{sq/2}) \mapsto f_\gamma(a_1,...,a_{sq/2}),
\]
of $sq/2$ variables, as follows:
\begin{itemize}
\item[(i)] juxtapose $s$ copies of $f$, and
\item[(ii)] if the vertices $j$ and $l$ are linked by $r$ edges, then identify $r$ variables in the argument of the $j$th copy of $f$ with $r$ variables in the argument of the $l$th copy. By symmetry, the position of the identified $r$ variables is immaterial. Also, by connectedness, one has necessarily $r<q$.
\end{itemize}
The resulting function $f_\gamma$ is a (not necessarily symmetric) element of \[
L^1(\mathbb{A}^{(sq/2)},\mathscr{A}^{(sq/2)},\mu^{(sq/2)}).
\]
For instance, if $\gamma =\{(1,2),(2,3),(3,1)\}\in K(3,2)$, then $f_\gamma (x,y,z) = f(x,y)f(y,z)f(z,x)$. If $\gamma =\{(1,2),(1,2),(2,3),(2,3),(1,3),(1,3)\}\in K(3,4)$, then
\[
f_\gamma (t,u,v,x,y,z) = f(t,u,v,x)f(t,u,y,z)f(y,z,v,x).
\]
}
\end{definition}

\medskip

We now turn to the main statement of this section, relating graphs and cumulants. The first part is classic (see e.g. \cite{PecTaq_SURV} for a proof), and shows how to use the functions $f_\gamma$ to compute the cumulants
of the random variable $F = I_q(f)$. The second part of the statement makes use of (\ref{formula-cumulants}), and shows indeed that Theorem \ref{thm-pasmal} implicitly provides a compact representation of well-known combinatorial expressions.

\begin{prop}
Let $q\geq 2$ and assume that $F=I_q(f)$, where $f\in\HH^{\odot q}$. Assume the integer $s\geq 2$ is such that $sq$ is even.
Then,
\begin{equation}\label{diagramform}
\kappa_s(F) = \sum_{\gamma\in K(s,q)}w(\gamma) \int_{\mathbb{A}^{(sq/2)}}f_\gamma(a_1,...,a_{sq/2})\mu(da_1)\cdot\cdot\cdot\mu(da_{sq/2}).
\end{equation}
As a consequence, by using (\ref{formula-cumulants}), one deduces the identity
\begin{eqnarray}
&&(s-1)!q!\sum c_q(r_1,\ldots,r_{s-2})\big\langle (...(f\widetilde{\otimes}_{r_1} f) \widetilde{\otimes}_{r_2} f)
\ldots
f)\widetilde{\otimes}_{r_{s-2}}f,f\big\rangle_{\HH^{\otimes q}} \\
&& = \sum_{\gamma\in K(s,q)}w(\gamma) \int_{\mathbb{A}^{(sq/2)}}f_\gamma(a_1,...,a_{sq/2})\mu(da_1)\cdot\cdot\cdot\mu(da_{sq/2}).
\end{eqnarray}
where $\sum$ runs over all collections of integers $r_1,\ldots,r_{s-2}$ verifying the conditions pinpointed in the statement of Theorem \ref{thm-pasmal}.
\end{prop}

\begin{rem}{\rm
Being based on a sum over the whole set $K(s,q)$, formula (\ref{diagramform}) is of course more compact than (\ref{formula-cumulants}). However, since it does not contain any hint about how one should enumerate the elements of $K(s,q)$, expression (\ref{diagramform}) is much harder to compute and asses. On the other hand, (\ref{formula-cumulants}) is based on recursive relations and inner products of contractions, so that one could in principle compute $\kappa_s(F)$ by implementing an explicit algorithm.
}
\end{rem}

\subsection{CLTs on Wiener chaos}
We conclude the paper by providing a new proof (based on our new
formula (\ref{formula-cumulants})) of the following result, first proved in \cite{nunugio} and
yielding a necessary and sufficient condition for CLTs on a fixed chaos.
\begin{thm}[See \cite{nunugio}]\label{thm-nupec}
Fix an integer $q\geq 2$, and let $(F_n)_{n\geq 1}$ be a sequence of the form $F_n=I_q(f_n)$,
with $f_n\in\HH^{\odot q}$ such that
$E[F_n^2]=q!\|f_n\|^2_{\HH^{\otimes q}}=1$ for all $n\geq 1$.
Then, as $n\to\infty$, we have $F_n\to N(0,1)$ in law if and only if
$\|f_n\widetilde{\otimes}_r f_n\|_{\HH^{\otimes (2q-2r)}}\to 0$
for all $r=1,\ldots,q-1$.
\end{thm}
{\it Proof}. Observe that $\kappa_1(F_n)=0$ and $\kappa_2(F_n)=1$.
For
$\kappa_s(F_n)$, $s\geq 3$, we consider the expression (\ref{formula-cumulants}).
Let $r_1,\ldots,r_{s-2}$ be some integers such that
$(i)$--$(iv)$ in Theorem \ref{thm-pasmal} are satisfied.
Using Cauchy-Schwarz inequality and then successively
$\|g\widetilde{\otimes}_r h\|_{\HH^{\otimes (p+q-2r)}}\leq
\|g\otimes_r h\|_{\HH^{\otimes (p+q-2r)}}\leq
\|g\|_{\HH^{\otimes p}}\|h\|_{\HH^{\otimes q}}$
whenever $g\in\HH^{\odot p}$, $h\in\HH^{\odot q}$
and $r=1,\ldots,p\wedge q$, we get
that
\begin{eqnarray}
&&\big|\langle (...(f_n\widetilde{\otimes}_{r_1} f_n) \widetilde{\otimes}_{r_2} f_n)\ldots
f_n)\widetilde{\otimes}_{r_{s-2}}f_n,f_n\rangle_{\HH^{\otimes q}}\big|\notag\\
&&\quad\quad\quad\quad\quad\quad\quad\quad\quad \quad\quad\quad\leq
\| (...(f_n\widetilde{\otimes}_{r_1} f_n) \widetilde{\otimes}_{r_2} f_n)\ldots
f_n)\widetilde{\otimes}_{r_{s-2}}f_n\|_{\HH^{\otimes q}}\|f_n\|_{\HH^{\otimes q}}\notag\\
&& \quad\quad\quad\quad\quad\quad\quad\quad\quad\quad\quad\quad  \leq\|f_n\widetilde{\otimes}_{r_1} f_n\|_{\HH^{\otimes (2q-2r_1)}}
\|f_n\|_{\HH^{\otimes q}}^{s-2}\notag\\
&& \quad\quad\quad \quad\quad\quad\quad\quad\quad \quad\quad\quad =(q!)^{1-\frac{s}2}\,\|f_n\widetilde{\otimes}_{r_1} f_n\|_{\HH^{\otimes (2q-2r_1)}}.
\label{majoration}
\end{eqnarray}
Assume now that $\|f_n\widetilde{\otimes}_r f_n\|_{\HH^{\otimes (2q-2r)}}\to 0$
for all $r=1,\ldots,q-1$, and fix an integer $s\geq 3$. By combining (\ref{formula-cumulants})
with (\ref{majoration}), we get that
$\kappa_s(F_n)\to 0$ as $n\to\infty$. Hence, applying the method of moments or cumulants, we get
that $F_n\to N(0,1)$ in law. Conversely, assume that $F_n\to N\sim N(0,1)$ in law.
Since the sequence $(F_n)$ lives inside the $q$th chaos, and because $E[F_n^2]=1$ for all $n$,
we have that, for every $p\geq 1$, $\sup_{n\geq 1}E[|F_n|^p]<\infty$
(see e.g. Janson \cite[Ch. V]{Janson}).
This implies immediately that $\kappa_4(F_n)=E[F_n^4]-3\to E[N^4]-3=0$. Hence,
identity (\ref{bob}) allows to conclude
that $\|f_n\widetilde{\otimes}_r f_n\|_{\HH^{\otimes (2q-2r)}}\to 0$ for all $r=1,\ldots,q-1$.
\fin

\end{document}